\documentclass[12pt,leqno]{amsart}
\usepackage{latexsym,amsmath,amssymb,mathrsfs}
\usepackage{hyperref}
\usepackage{esint}
\usepackage{mathtools}
\DeclarePairedDelimiterX{\norm}[1]{\lVert}{\rVert}{#1}
\usepackage{enumitem}
\usepackage{arydshln}

\usepackage{tikz}

\title[Trudinger inequality]{The Trudinger inequality is true for $M^{1,s}$ spaces} 

\author[R.~Alvarado]{Ryan Alvarado}
\address{Ryan Alvarado: Department of Mathematics, Amherst College, Amherst, MA 01002, USA} 
\email{\tt rjalvarado@amherst.edu}

\author[A.~Dughayshim]{Ahmed Dughayshim}
\address{Ahmed Dughayshim: Department of Mathematics, University of Pittsburgh, 301
Thackeray Hall, Pittsburgh, PA 15260, USA} 
\email{\tt AHA80@pitt.edu}

\author[P.~Haj\l{}asz]{Piotr Haj\l{}asz}
\address{Piotr Haj{\l}asz: Department of Mathematics, University of Pittsburgh, 301
Thackeray Hall, Pittsburgh, PA 15260, USA} 
\email{\tt hajlasz@pitt.edu}
\thanks{P.H. was supported by NSF grant  DMS-2452426.}

plus 6pt minus 12pt
\newtheorem{theorem}{Theorem}
\newtheorem{lemma}[theorem]{Lemma}
\newtheorem{corollary}[theorem]{Corollary}

\newtheorem*{proposition*}{Proposition}

\theoremstyle{definition}
\newtheorem{remark}[theorem]{Remark}

\newcommand{\barint}{
\rule[.036in]{.12in}{.009in}\kern-.16in \displaystyle\int }

\newcommand{\barcal}{\mbox{$ \rule[.036in]{.11in}{.007in}\kern-.128in\int $}}
\newcommand{\bbbz}{\mathbb Z}
\newcommand{\bbbr}{\mathbb R}

\def\diam{\operatorname{diam}}

\def\H{{\mathcal H}}

\def\mvint_#1{\mathchoice
          {\mathop{\vrule width 6pt height 3 pt depth -2.5pt
                  \kern -8pt \intop}\nolimits_{\kern -3pt #1}}%
          {\mathop{\vrule width 5pt height 3 pt depth -2.6pt
                  \kern -6pt \intop}\nolimits_{#1}}%
          {\mathop{\vrule width 5pt height 3 pt depth -2.6pt
                  \kern -6pt \intop}\nolimits_{#1}}%
          {\mathop{\vrule width 5pt height 3 pt depth -2.6pt
                  \kern -6pt \intop}\nolimits_{#1}}}

\numberwithin{theorem}{section} \numberwithin{equation}{section}

\begin{document}
\sloppy



\begin{abstract}
We prove the Trudinger inequality with exponent $\frac{s}{s-1}$ for $M^{1,s}$ Sobolev spaces on metric measure spaces under the sole measure growth assumption $\mu(B(x,r))\ge br^s$. This improves the previously known exponential integrability with power one. Our argument also yields sharp asymptotic bounds for the Sobolev constants as $p\uparrow s$. On doubling spaces, we further remove the connectedness assumption from the Trudinger inequality associated with a $(1,p)$-Poincaré inequality when $p<s$. An Ahlfors regular counterexample shows that this extension fails at the critical exponent $p=s$.
\end{abstract}

\maketitle

\section{Introduction}

To begin, we define a metric measure space $(X,d,\mu )$ as a metric space $(X,d)$ equipped with a Borel measure $\mu$ that satisfies $0<\mu(B(x,r))<\infty$ for any $x\in X$, $r\in(0,\infty)$. Further, we will always assume that $X$ has at least two points.

The study of first order Sobolev spaces on general metric measure spaces was first initiated in \cite{hajlasz2}. In the current literature these spaces are denoted by $M^{1,p}(X)$. For $0<p\leq\infty$, the space $M^{1,p}(X)$ consists of all $u\in L^p(X)$ such that there is $0\leq g\in L^p(X)$ satisfying
\begin{equation}
\label{eq14}
|u(x)-u(y)|\leq d(x,y)(g(x)+g(y))
\qquad 
\text{almost everywhere,}
\end{equation}
i.e., for all $x,y\in X\setminus N$, where $\mu(N)=0$.

The main result of \cite[Theorem~6]{hajlasz2} was the Sobolev embedding theorem under the assumption that the space is bounded and the measure satisfies the lower bound $\mu(B(x,r))\geq br^s$; see \eqref{eq1}. It turned out that the exponent $s$ plays the role of the dimension of the space, leading to different embeddings according to whether $0<p<s$, $p=s$, or $p>s$. The result was later generalized to local estimates on possibly unbounded spaces in \cite[Theorem~8.7]{hajlasz}; see also \cite[Theorem~6]{AGH}. These local estimates also yield H\"older continuity when $p>s$, a conclusion that was not obtained in \cite{hajlasz2}. Thus, roughly speaking, these papers established the following counterparts of the classical Sobolev embeddings for $u\in M^{1,p}$: if $0<p<s$, then $u\in L^{p^*}$, where $p^*=sp/(s-p)$; if $p=s$, then $\exp(Cu)\in L^1$; and if $s<p<\infty$, then $u\in C^{0,1-s/p}$.

After \cite{hajlasz2}, several other classes of Sobolev spaces on metric measure spaces were introduced; see, for example, \cite{ambrosiogs,cheeger,hajlasz,SMP,shanmugalingam}. In these approaches, proofs of embedding theorems typically require the measure to be doubling. In contrast, for $M^{1,p}$ spaces it suffices to assume only the lower bound $\mu(B(x,r))\geq br^s$. For this reason, in this paper we focus primarily on $M^{1,p}$ spaces; see, however, Theorems~\ref{T9} and~\ref{T3}.

The proofs of the embedding theorems from \cite{hajlasz,hajlasz2} were subsequently included in the monographs \cite[Theorem~5.4.2]{AT}, \cite[Theorem~8.6 (Theorem~12.1 in the preprint version)]{DGN}, and \cite[Theorem~8.1.18]{HKST}. The theory was later extended to other classes of function spaces, including Triebel--Lizorkin and Besov spaces; see, for example, \cite{AYY21,AYY24,GKZ,KosYZ,KosYZ2}. Moreover, \cite{hajlasz2} has been cited hundreds of times, with its results being applied, refined, and generalized in a wide variety of settings.

While the embeddings $u\in L^{p^*}$ for $0<p<s$, and $u\in C^{0,1-s/p}$ for $s<p<\infty$ have sharp exponents, the borderline embedding $\exp(C|u|)\in L^1$ for $p=s$ is weaker than the Trudinger inequality available in the classical setting. Thus, one is naturally led to conjecture that if $p=s>1$, then $\exp(C|u|^{s'})\in L^1$, where $s'=s/(s-1)$.

Despite the fact that hundreds of papers have been devoted to various generalizations of the Trudinger inequality, there has been essentially no progress on the question of whether the Trudinger inequality holds for $M^{1,s}$ spaces, with two exceptions that we discuss below.

If additionally the measure is doubling, $p=s>1$, and the space is connected, the Trudinger inequality for $M^{1,p}$ spaces follows from \cite[Theorem~6.1]{SMP}, because $M^{1,p}$ spaces satisfy the Poincar\'e inequality, see Corollary~\ref{T5} below. However, for $M^{1,p}$ spaces we only assume that $\mu(B(x,r))\geq b r^s$, so no doubling or connectedness. Also in the paper \cite[Theorem~1.1 and (6)]{GK} the authors prove the Trudinger inequality for $M^{1,s}$ spaces in $\bbbr^n$ under the assumption that $\mu(B^n(x,r))\geq br^s$, $s>1$, and that the measure vanishes on all hyperplanes perpendicular to coordinate axes. While the doubling condition is not assumed here, the proof is very Euclidean. In particular it uses the nested families of cubes and the Besicovitch covering lemma. 

If $0<s<1$, then formally $s/(s-1)<0$, so one would expect that instead of the Trudinger inequality, in the case $0<p=s\leq 1$ we get continuity of $u$. In fact in \cite{zhou} it was proved that functions in $M^{1,s}$ are continuous whenever $X$ is Ahlfors $s$-regular. Again, the measure is doubling and in fact the space is uniformly perfect which is somewhat related to connectivity, see also \cite{Bondarev,GorkaS}.

The aim of this paper is to prove that in a general case functions in $M^{1,s}$ satisfy the Trudinger inequality. The proof is based on a modification of the original arguments introduced in \cite{hajlasz}. While the original argument couldn't be used to produce the Trudinger inequality, some estimates are much sharper now, due to a new idea that seemed to be overlooked in the literature. This argument gives also a sharp asymptotic of the constant in the Sobolev embedding when $0<p<s$.

Let us denote by $D(u)$ the class of measurable functions $g\geq 0$ satisfying \eqref{eq14}. 
By $f_A=\fint_A f\, d\mu:=\mu(A)^{-1}\int_A f\, d\mu$ we denote the integral average of $f$.

The first main result in the paper is the following improvement of \cite[Theorem~6]{hajlasz2}.
\begin{theorem}
\label{T1}
Let $(X,d)$ be bounded, $0<\mu(X)<\infty$, and assume that for some $s>0$ and $b>0$,
\begin{equation}
\label{eq1}
 \mu(B(x,r))\ge br^s,
 \qquad x\in X,\quad 0<r\leq\diam X.
\end{equation}
Let $u\in M^{1,p}(X)$ and let $g\in D(u)\cap L^p(X)$.
\begin{enumerate}
\item If $0<p<s$ and $p^*=\frac{sp}{s-p}$,
then
\begin{equation}
\label{eq57}
 \inf_{\gamma\in\mathbb R}\Vert u-\gamma\Vert_{p^*}
 \leq C_{s,p}b^{-1/s}\Vert g\Vert_p.
\end{equation}
If $p^*\ge1$, then $u_X$ is defined and one may take $\gamma=u_X$, after changing the constant to $2C_{s,p}$. If $1\leq p<s$, one can take in \eqref{eq57}
\begin{equation}
\label{eq16}
C_{s,p}=64\cdot 2^s\left(\frac{s-1}{s-p}\right)^{1/s'},
\end{equation}
\item If $p=s>1$, then
\begin{equation}
\label{eq17}
\fint_X
 \exp\!\left[
 c_s\left(
 \frac{b^{1/s}|u-u_X|}{\Vert g\Vert_{s}}
 \right)^{\frac{s}{s-1}}
 \right]d\mu
 \leq C_s.
\end{equation}
If $p=s$ and $0<s\leq1$, then
\[
 \Vert u-u_X\Vert_{\infty}
 \leq C_s b^{-1/s}\Vert g\Vert_{s}.
\]

\item If $s<p<\infty$, then
\[
 \Vert u-u_X\Vert_{\infty}
 \leq C_{s,p}b^{-1/s}\mu(X)^{1/s-1/p}\Vert g\Vert_p.
\]
\end{enumerate}
\end{theorem}
%\begin{remark}
%The proof differs from the original proof of Theorem~6 only at the point where the increments in \eqref{eq5} below are summed. Instead of estimating each level separately by Chebyshev's inequality, we retain the stronger summability estimate \eqref{eq8} and use elementary embeddings between the sequence spaces $\ell^r$. In the critical case $s>1$ this is the finite-dimensional estimate
%\[
% \Vert (a_j)_{j=1}^N\Vert_{\ell^1}
% \leq N^{1/s'}\Vert (a_j)_{j=1}^N\Vert_{\ell^s},
%\]
%while for $0<s\leq1$ one uses $\ell^s\subset\ell^1$. A weighted version of the same argument gives the subcritical embedding for the full range $0<p<s$; no step requires $p>1$. This is the only essential modification, and in the critical case it gives the sharp Trudinger exponent $s'=s/(s-1)$.
%\end{remark}
\begin{remark}
One can prove the Trudinger inequality \eqref{eq17}  directly using the sharp asymptotic \eqref{eq16}. However, our proof is different.
\end{remark}
\begin{remark}
In \eqref{eq17}, it is assumed that $\Vert g\Vert_s>0$. A similar assumption is implicit in the other statements of Trudinger-type inequalities appearing in the paper.
\end{remark}

The above result is the global embedding theorem for bounded spaces. The next result is a local version of the Sobolev embedding theorem which is a counterpart of \cite[Theorem~8.7]{hajlasz} and \cite[Theorem~6]{AGH}. 

For a real-valued function $v$ on a set $A$ we write
\[
 \operatorname{osc}_A v=\sup_Av-\inf_Av.
\]
Also, if $\sigma\geq 1$, by $\sigma B$ we denote a ball concentric with $B$ and with $\sigma$ times the radius.
\begin{theorem}
\label{T2}
Let $0<p<\infty$, $1<\sigma\leq2$, and let $B=B(z,R)$ be any ball of radius $R$. Assume that for some $s,b,R_o>0$,
\begin{equation}
\label{eq20}
 \mu(B(x,r))\geq br^s,
 \qquad x\in\sigma B,\quad 0<r\leq R_o.
\end{equation}
Set
\[
 L=\min\{R_o,(\sigma-1)R\}.
\]
Let $u\in M^{1,p}(\sigma B)$ and let $g\in D(u)\cap L^p(\sigma B)$.
\begin{enumerate}
\item If $0<p<s$ and $p^*=sp/(s-p)$, then
\[
 \inf_{\gamma\in\bbbr}
 \left(\fint_B|u-\gamma|^{p^*}\,d\mu\right)^{1/p^*}
 \leq C_{s,p}\frac{R}{L}\,b^{-1/p}L^{1-s/p}
 \Vert g\Vert_{L^p(\sigma B)}.
\]
If $p^*\geq1$, then one may take $\gamma=u_B$, after changing the constant. If $1\leq p<s$, then, with $s'=s/(s-1)$, we get the sharper estimate
\begin{equation}
\label{eq54}
 \left(\fint_B|u-u_B|^{p^*}\,d\mu\right)^{1/p^*}
 \leq C_s\left(
 \frac{R}{L}+\left(\frac{s-1}{s-p}\right)^{1/s'}
 \right)b^{-1/p}L^{1-s/p}
 \Vert g\Vert_{L^p(\sigma B)}.
\end{equation}

\item If $p=s>1$, then
\[
 \fint_B
 \exp\!\left[
 c_s\left(
 \frac{L}{R}\frac{b^{1/s}|u-u_B|}
 {\Vert g\Vert_{L^s(\sigma B)}}
 \right)^{\frac{s}{s-1}}
 \right]d\mu
 \leq C_s.
\]
If $p=s\leq1$, then $u|_B$ has a  uniformly continuous representative that satisfies
\begin{equation}
\label{eq osc}
 \operatorname{osc}_B u
 \leq C_s\frac{R}{L}\,b^{-1/s}
 \Vert g\Vert_{L^s(\sigma B)}.
\end{equation}

\item If $p>s$, then $u|_B$ has a H\"older continuous representative of order $1-s/p$ that satisfies
\[
 \operatorname{osc}_B u
 \leq C_{s,p}\frac{R}{L}\,b^{-1/p}L^{1-s/p}
 \Vert g\Vert_{L^p(\sigma B)},
\]
and
\[
 |u(x)-u(y)|
 \leq C_{s,p}\frac{R}{L}\,b^{-1/p}d(x,y)^{1-s/p}
 \Vert g\Vert_{L^p(\sigma B)}
\]
for all $x,y\in B$.
\end{enumerate}
\end{theorem}

\begin{remark}
In the special case $\sigma=2$, if $R\leq R_o$ and $1\leq p<s$,
estimate \eqref{eq54} yields
\begin{equation}
\label{eq55}
\left(\fint_B |u-u_B|^{p^*}\, d\mu\right)^{1/p^*}\leq
C_s\left(\frac{s-1}{s-p}\right)^{1/s'} b^{-1/p} R^{1-s/p}\Vert g\Vert_{L^p(2B)}.
\end{equation}
This gives the same asymptotic as in Theorem~\ref{T1} when $p\uparrow s$. One can also show that under the assumptions of Theorem~\ref{T1} when $1\leq p<s$, \eqref{eq54} gives
\begin{equation}
\label{eq56}
\Vert u-u_X\Vert_{p^*}\leq C_s\left(\frac{s-1}{s-p}\right)^{1/s'}b^{-1/s}\Vert g\Vert_p,
\end{equation}
and Theorem~\ref{T1} gives the same inequality with $C_s=128\cdot 2^s$.

The deduction of \eqref{eq55} is immediate, whereas \eqref{eq56},
although still elementary, requires an observation  that \eqref{eq20} is in fact true for a possibly larger range of radii, $0<r\leq \rho:=(\mu(X)/b)^{1/s}$, where clearly $\rho\geq\diam X$. We leave details to the reader.
\end{remark}

\begin{remark}
% In the case $p=s\leq1$, Zhou \cite{zhou} obtained uniform continuity under the assumption that the measure is Ahlfors regular. In that setting, the same conclusion also follows from \eqref{eq osc} and the absolute continuity of the integral, since the upper measure bound implies that the measures of balls tend uniformly to zero as their radii tend to zero. Since Theorem~\ref{T2} only assumes a lower measure bound, the
% measure may have atoms, so \eqref{eq osc} alone does not directly imply uniform continuity.
In the case $p=s\leq1$, Zhou \cite{zhou} obtained uniform continuity under the assumption that the measure is Ahlfors regular. In that setting, the upper measure bound ensures that the measures of balls tend uniformly to zero as their radii tend to zero; hence, uniform continuity would follow from \eqref{eq osc} and the absolute continuity of the integral. With only the lower measure bound assumed in Theorem~\ref{T2}, the measure may have atoms, so \eqref{eq osc} does not directly give uniform continuity.
\end{remark}

Although Theorem~\ref{T2} implies Theorem~\ref{T1} and, in fact,
yields stronger conclusions when $0<s=p\leq1$ or $p>s$, we include
a detailed proof of Theorem~\ref{T1} for several reasons.
Its statement is easier to grasp at first glance, and its proof is technically much simpler. This simplicity allows the reader to see clearly the main novelty of our approach. It is particularly useful for readers who do not wish to go through all the technical details of the paper but would still like to understand its main idea. Moreover, understanding the proof of Theorem~\ref{T1} makes it easier to follow the more technical proof of Theorem~\ref{T2}. In the proof of Theorem~\ref{T2}, arguments that closely parallel those used in the proof of Theorem~\ref{T1} will therefore be treated more briefly, allowing us to focus on the genuinely new aspects of the argument.

The results of the paper apply also to the case of the Trudinger inequality for Poincar\'e inequalities on doubling spaces. 

Recall that a measure $\mu$ is called doubling if there exists $C_d\geq 1$ such that $0<\mu(2B)\leq C_d\mu(B)$ for every ball $B\subset X$. 

We will also need the following growth condition from \cite[(21)]{SMP}: there is $C>0$ and $s>0$ such that
\begin{equation}
\label{eq18}
\frac{\mu(B)}{\mu(B_0)}
\geq C\left(\frac{r}{r_0}\right)^s,
\end{equation}
whenever $B_0$ is an arbitrary ball of radius $r_0$ and $B=B(x,r)$, $x\in B_0$, $0<r\leq r_0$.

Standard iteration of the doubling condition implies that \eqref{eq18} is always true with $s=\log_2 C_d$ and $C=4^{-s}$, see for example \cite[Lemma~4.7]{hajlasz}. However, it may happen that a doubling measure satisfies \eqref{eq18} with an exponent smaller than $\log_2 C_d$.

Given $s>0$, and $0<p<\infty$, we say that a pair $(u,g)$, where $u\in L^1_{\rm loc}(X)$, $0\leq g\in L^p_{\rm loc}(X)$ satisfies an
$(1,p)$-Poincar\'e inequality, with parameters $\sigma\geq 1$ and
$C_{PI}>0$ if
\[
\fint_B |u-u_B|\,d\mu
\leq C_{PI}\operatorname{diam}(B)
\left(\fint_{\sigma B} g^p\,d\mu\right)^{1/p}
\]
for every ball $B\subset X$. 

The following theorem was proved in \cite[Theorem~6.1]{SMP}.
\begin{theorem}
\label{T4}
Let $s\in (1,\infty)$. Assume that $(X,d,\mu)$ is connected, $\mu$ is doubling,
and \eqref{eq18} holds. Suppose that the pair $(u,g)$ satisfies
a $(1,s)$-Poincar\'e inequality with parameters $\sigma$ and
$C_{PI}$. Then there exist constants $c,C>0$ such that
\[
\fint_B
\exp\left[
c\left(
\frac{\mu(B)^{1/s}|u-u_B|}
{\operatorname{diam}(B)\,
\|g\|_{L^s(5\sigma B)}}
\right)^{\frac{s}{s-1}}
\right]\,d\mu
\leq C
\]
for every ball $B\subset X$.
\end{theorem}
This result easily implies the following Trudinger inequality for $M^{1,s}$ spaces.
\begin{corollary}
\label{T5}
Let $s\in (1,\infty)$. Assume that $(X,d,\mu)$ is connected, $\mu$ is doubling,
and \eqref{eq18} holds. If $u\in M^{1,s}(X)$ and $g\in D(u)\cap L^s(X)$, then there exist constants $c,C>0$ such that
\begin{equation}
\label{eq19}
\fint_B
\exp\left[
c\left(
\frac{\mu(B)^{1/s}|u-u_B|}
{\operatorname{diam}(B)\,
\|g\|_{L^s(5B)}}
\right)^{\frac{s}{s-1}}
\right]\,d\mu
\leq C
\end{equation}
for every ball $B\subset X$.
\end{corollary}
Indeed, if $B\subset X$ is a ball, then \eqref{eq14} and H\"older's inequality give,
\[
\fint_B|u-u_B|\,d\mu
\leq 2\diam B \fint_Bg\,d\mu
\leq 2\diam B \left(\fint_Bg^s\,d\mu\right)^{1/s},
\]
and \eqref{eq19} follows directly from Theorem~\ref{T4}.

Theorem~\ref{T4} requires connectedness in order to obtain the Trudinger inequality from a $(1,s)$-Poincar\'e inequality. The next result shows that connectedness can be dispensed with if the Poincar\'e inequality is available with some exponent $p<s$. Thus, below the critical exponent, the conclusion of Theorem~\ref{T4} remains valid on arbitrary doubling spaces satisfying \eqref{eq18}.

\begin{theorem}
\label{T9}
Let $s\in (1,\infty)$ and $0<p<s$. Assume that $(X,d,\mu)$ is doubling and that \eqref{eq18} holds. Suppose that $u,g\in L^s_{\rm loc}(X)$, $g\geq0$, and that the pair $(u,g)$ satisfies a $(1,p)$-Poincar\'e inequality with parameters $\sigma\geq1$ and $C_{PI}>0$. Then there exist constants $c,C>0$ such that
\begin{equation}
\label{eq50}
 \fint_B
 \exp\left[
 c\left(
 \frac{\,\mu(B)^{1/s}|u-u_B|}
 {R\,\Vert g\Vert_{L^s(10\sigma B)}}
 \right)^{\frac{s}{s-1}}
 \right]d\mu
 \leq C,
\end{equation}
for every ball $B=B(z,R)$. 
\end{theorem}

The restriction $p<s$ in Theorem~\ref{T9} is essential. In fact, the next result shows that if one replaces $p<s$ by the critical exponent $p=s$, then without connectedness the situation changes completely: even on an Ahlfors $s$-regular compact space, a $(1,s)$-Poincar\'e inequality need not imply exponential integrability with any power $q>1$. Thus Theorem~\ref{T9} gives precisely the improvement of Theorem~\ref{T4} that is possible without connectedness.

\begin{theorem}
\label{T3}
Let $s\in(1,\infty)$. Then there exists an Ahlfors $s$-regular compact metric measure space $(X,d,\mu)$ with the following property: there exists a pair $(u,g)$ with $u,g\in L^s(X)$ satisfying a $(1,s)$-Poincar\'e inequality for all balls in $X$, but
\begin{equation}
\label{eq40}
\int_X \exp\big(c|u-u_X|^q\big)d\mu=\infty
\end{equation}
for every $q>1$ and every $c>0$.
\end{theorem}

Theorem~\ref{T4} was extended in \cite{MP} to a much broader class of functionals satisfying a certain $T_p$ condition. They proved a Trudinger-type inequality under the assumption that the space is connected, and in Section~4 they showed that, in general, the conclusion fails without connectedness. However, their counterexample does not imply Theorem~\ref{T3}, because the functional used there is not associated with a $(1,s)$-Poincar\'e inequality. In fact, the construction of their counterexample is very simple precisely because of the considerable freedom allowed in the choice of the functional.

\subsection*{Notation}
We use $C$ and $c$ to denote generic positive constants whose values may change from line to line. Typically, $C$ denotes a large constant and $c$ a small one. Subscripts indicate dependence on parameters; for example, $C_{s,p}$ and $c_{s,p}$ depend only on $s$ and $p$. We write $\Vert g\Vert_p$ for the $L^p$ norm on $X$ and $\Vert g\Vert_{L^p(E)}$ for the $L^p$ norm on a measurable subset $E\subset X$.
By $D(u)$ we denote the class of measurable functions $g\geq 0$ satisfying \eqref{eq14}. 
By $f_A=\fint_A f\, d\mu:=\mu(A)^{-1}\int_A f\, d\mu$ we denote the integral average of $f$.

\subsection*{Declaration on the use of generative AI} 
The authors used OpenAI's ChatGPT to check proofs, search literature and improve the exposition. All mathematical arguments, results, and references appearing in the final manuscript were written and verified by the authors, who take full responsibility for the content of the paper.
The work of Piotr Hajłasz was supported in part by a grant of access to OpenAI models through the ChatGPT for Academic Researchers program.
\section{Global estimates in bounded spaces}

\begin{proof}[Proof of Theorem~\ref{T1}]
We may assume that $\Vert g\Vert_p>0$ and that, for some null set $N$,
\[
 |u(x)-u(y)|\leq d(x,y)(g(x)+g(y))
\]
for every $x,y\in X\setminus N$. Since
\[
 \mu(X)\ge b(\diam X)^s,
\]
we have
\begin{equation}
\label{eq2}
\diam X\leq b^{-1/s}\mu(X)^{1/s}.
\end{equation}
Let
\[
 E_k=\left\{x\in X\setminus N:\ g(x)\leq 2^k\mu(X)^{-1/p}\Vert g\Vert_p\right\},
 \qquad
 m_k=\mu(X\setminus E_k).
\]
Note that it is an increasing family of sets, $E_k\subset E_{k+1}$.
Enlarging $N$ if necessary, we may assume that $E_j=X\setminus N$ whenever $m_j=0$, and hence $E_k=X\setminus N$ for all $k\geq j$.

Chebyshev's inequality gives
\begin{equation}
\label{eq3}
m_k\leq 2^{-kp}\mu(X).
\end{equation}
Let $k_o$ be the smallest positive integer such that
\[
 k_op>1.
\]
Then \eqref{eq3} implies $m_{k_o}<\mu(X)/2$, so $E_{k_o}\ne\varnothing$. Fix $x_o\in E_{k_o}$ and put
\[
 a_k=\sup_{E_k}|u-u(x_o)|,
 \qquad k\ge k_o.
\]
On $E_{k_o}$ we have
\[
 |u(x)-u(y)|\leq 2^{k_o+1}\mu(X)^{-1/p}\Vert g\Vert_p d(x,y),
\]
so $u$ is Lipschitz continuous on $E_{k_o}$, and by \eqref{eq2},
\begin{equation}
\label{eq4}
a_{k_o}
\leq 4\cdot 2^{1/p} b^{-1/s}\mu(X)^{1/s-1/p}\Vert g\Vert_p,
\end{equation}
because $(k_o-1)p\leq 1$.

We claim that 
\begin{equation}
\label{eq9}
a_k\leq a_{k_o}
+2^{1+1/s}b^{-1/s}\mu(X)^{-1/p}\Vert g\Vert_p\sum_{i=k_o+1}^k 2^im_{i-1}^{1/s},
\qquad
k> k_o.
\end{equation}
To prove it, it suffices to show that
\begin{equation}
\label{eq5}
a_k\leq a_{k-1}
+2^{1+1/s}b^{-1/s}\mu(X)^{-1/p}\Vert g\Vert_p 2^km_{k-1}^{1/s},
\end{equation}
because \eqref{eq9} will follow from a simple iteration applied to \eqref{eq5}.

If $m_{k-1}=0$, then $E_{k-1}=E_k=X\setminus N$; thus $a_k=a_{k-1}$ and \eqref{eq5} holds, so we can assume that $m_{k-1}>0$. Let
\begin{equation}
\label{eq6}
 r_k=\left(\frac{2m_{k-1}}b\right)^{1/s}.
\end{equation}
Let $x\in E_k$. We claim that
\begin{equation}
\label{eq7}
d(x,y)\leq r_k
\qquad
\text{for some } y\in E_{k-1}.
\end{equation}
This is obvious if $r_k\geq \diam X$, because $E_{k-1}\neq\varnothing$, and if $0<r_k<\diam X$, \eqref{eq1} yields
\[
\mu(B(x,r_k))\ge 2m_{k-1},
\quad
\text{so}
\quad
B(x,r_k)\cap E_{k-1}\neq\varnothing,
\]
proving \eqref{eq7}.
Since $x,y\in E_k$, $d(x,y)\leq r_k$, and $y\in E_{k-1}$, we have
\[
|u(x)-u(x_o)|\leq |u(y)-u(x_o)|+|u(x)-u(y)|\leq a_{k-1}+r_k2^{k+1}\mu(X)^{-1/p}\Vert g\Vert_p
\]
and \eqref{eq5} follows from \eqref{eq6} upon taking supremum over $x\in E_k$.

We shall need the following fact
\begin{equation}
\label{eq8}
\sum_{k\in\mathbb Z}
2^{kp}m_{k-1}
\leq \frac{2^p}{1-2^{-p}} \mu(X).
\end{equation}
Indeed, Fubini's theorem yields
\[
\begin{split}
LHS
&=
\sum_{k\in\bbbz}\int_X 2^{kp}\chi_{\{x:g(x)>2^{k-1}\mu(X)^{-1/p}\Vert g\Vert_p}\}\, d\mu(x)
=
\int_X\sum_{\{ k:\, 2^k<2g(x)\mu(X)^{1/p}\Vert g\Vert_p^{-1}\}}2^{kp}\, d\mu(x)\\
&\leq
\int_X\frac{(2g(x)\mu(X)^{1/p}\Vert g\Vert_p^{-1})^p}{1-2^{-p}}\, d\mu(x)=\frac{2^p}{1-2^{-p}}\mu(X).
\end{split}
\]

\noindent
{\sc The subcritical case $0<p<s$.}
Our first aim is to prove that
\begin{equation}
\label{eq10}
a_k\leq \Gamma_{s,p}\,b^{-1/s}\mu(X)^{-1/p^*}\Vert g\Vert_p\,2^{kp/p^*},
\qquad k\ge k_o,
\end{equation}
and that 
\begin{equation}
\label{eq15}
\Gamma_{s,p}\leq 32\left(\frac{s-1}{s-p}\right)^{1/s'}
\qquad
\text{when } 1\leq p<s.
\end{equation}
We first assume that $s>1$ and write $s'=s/(s-1)$. Applying H\"older's inequality with exponents $s,s'$ to \eqref{eq9}, and using the estimate in \eqref{eq8}, we obtain
\begin{align*}
a_k
&\leq a_{k_o}
 +2^{1+1/s}b^{-1/s}\mu(X)^{-1/p}\Vert g\Vert_p
 \left(
 \sum_{i=k_o+1}^k
 2^{ip}m_{i-1}
 \right)^{1/s}\cdot
 \left(
 \sum_{i=k_o+1}^k
2^{i(1-\frac{p}{s})s'}\right)^{1/s'}
\end{align*}
Put
\[
\theta=\left(1-\frac ps\right)s'
=\frac{s-p}{s-1}>0.
\]
By \eqref{eq8},
\[
\left(\sum_{i=k_o+1}^k2^{ip}m_{i-1}\right)^{1/s}
\leq
\left(\frac{2^p}{1-2^{-p}}\right)^{1/s}\mu(X)^{1/s},
\]
while
\[
\left(\sum_{i=k_o+1}^k2^{i\theta}\right)^{1/s'}
\leq
\left(\frac{2^{k\theta}}{1-2^{-\theta}}\right)^{1/s'}
=
(1-2^{-\theta})^{-1/s'}2^{kp/p^*}.
\]
Consequently,
\[
a_k\leq a_{k_o}
+2^{1+\frac{1}{s}+\frac{p}{s}}(1-2^{-p})^{-1/s}(1-2^{-\theta})^{-1/s'}
 b^{-1/s}\mu(X)^{-1/p^*}\Vert g\Vert_p\,2^{kp/p^*}.
\]
Now, we will estimate $a_{k_o}$. Since $k\geq k_o$ and $k_op>1$, we have $kp/p^*>1/p^*=\frac{1}{p}-\frac{1}{s}$. Hence, \eqref{eq4} yields
\[
a_{k_o}\leq 2^{2+\frac{1}{s}}b^{-1/s}\mu(X)^{-1/p^*}\Vert g\Vert_p 2^{\frac{1}{p}-\frac{1}{s}}\leq
2^{2+\frac{1}{s}} b^{-1/s}\mu(X)^{-1/p^*}\Vert g\Vert_p 2^{kp/p^*},
\quad
k\geq k_o.
\]
Thus, for $s>1$ and $0<p<s$, \eqref{eq10} holds with
\[
 \Gamma_{s,p}
 =2^{3+1/s}(1-2^{-p})^{-1/s}
 \left(1-2^{-\theta}\right)^{-1/s'}.
\]
If $1\leq p<s$, then $1-2^{-p}\ge\frac12$ and 
$1-2^{-\theta}\ge \theta/2$, because $\theta\in(0,1]$, so 
\[
\Gamma_{s,p}
\leq 2^{4+1/s}
\left(\frac{s-1}{s-p}\right)^{1/s'}
\leq 32\left(\frac{s-1}{s-p}\right)^{1/s'}.
\]
This proves \eqref{eq15}. 

It remains to prove \eqref{eq10} when $0<s\leq 1$ (and $0<p<s$). Using
$\sum_j\alpha_j\leq\left(\sum_j\alpha_j^s\right)^{1/s}$
we get
\begin{equation}
\label{eq39}
\sum_{i=k_o+1}^k2^im_{i-1}^{1/s}
\leq
\left(\sum_{i=k_o+1}^k2^{is}m_{i-1}\right)^{1/s}
\leq
2^{kp/p^*}
\left(\sum_{i=k_o+1}^k2^{ip}m_{i-1}\right)^{1/s}
\leq C_{s,p}2^{kp/p^*}\mu(X)^{1/s}.
\end{equation}
We used here a simple observation that since $s-p>0$,
\[
2^{is}=2^{ip}2^{i(s-p)}\leq 2^{k(s-p)}2^{ip}=2^{kps/p^*}2^{ip}
\]
and the last inequality in \eqref{eq39} follows from \eqref{eq8}. Hence \eqref{eq9} yields
\[
a_k\leq a_{k_o}+C_{s,p}b^{-1/s}\mu(X)^{-1/p^*}\Vert g\Vert_p2^{kp/p^*}.
\]
The remaining term $a_{k_o}$ is absorbed using \eqref{eq4} and this completes the proof of \eqref{eq10}.

Now we can complete the proof of the embedding theorem.
Let
\[
 F_{k_o}=E_{k_o},
 \qquad
 F_k=E_k\setminus E_{k-1},\quad k>k_o.
\]
Up to a null set,
\[
 X=\bigcup_{k\geq k_o}F_k.
\]
Since $(k_o-1)p\leq1$ and $p<s$,
\[
 2^{k_op}\mu(F_{k_o})
 \leq 2^{p+1}\mu(X)
 \leq 2^{s+1}\mu(X).
\]
If $k>k_o$, then on $F_k$ we have
$2^{kp}<2^p\mu(X)\Vert g\Vert_p^{-p}g(x)^p$
and hence
\begin{equation}
\label{eq11}
 \sum_{k\geq k_o}2^{kp}\mu(F_k)
 \leq
2^{s+1}\mu(X)+ 
2^p\mu(X)\Vert g\Vert_p^{-p}
\sum_{k>k_o}\int_{F_k}g^p\,d\mu
 \leq 2^{s+2}\mu(X).
\end{equation}
Since $|u-u(x_o)|\leq a_k$ on $F_k$,
using \eqref{eq10} and \eqref{eq11}, we obtain
\[
\int_X|u-u(x_o)|^{p^*}\,d\mu
\leq \Gamma_{s,p}^{p^*}b^{-p^*/s}\mu(X)^{-1}\Vert g\Vert_p^{p^*}
\sum_{k\geq k_o}2^{kp}\mu(F_k)
\leq 2^{s+2}\Gamma_{s,p}^{p^*}b^{-p^*/s}\Vert g\Vert_p^{p^*}.
\]
Therefore,
\[
\inf_{\gamma\in\mathbb R}\Vert u-\gamma\Vert_{p^*}
\leq 2^{(s+2)/p^*} \Gamma_{s,p}b^{-1/s}\Vert g\Vert_p.
\]
If $1\leq p<s$, then $p^*\geq s'$ and hence \eqref{eq15} gives
\[
\inf_{\gamma\in\mathbb R}\Vert u-\gamma\Vert_{p^*}
\leq 2^{(s+2)/s'}\cdot 32 \left(\frac{s-1}{s-p}\right)^{1/s'}
b^{-1/s}\Vert g\Vert_p
\leq 64\cdot 2^s\left(\frac{s-1}{s-p}\right)^{1/s'}
b^{-1/s}\Vert g\Vert_p.
\]
Finally, if $p^*\geq1$, then 
$\Vert u-u_X\Vert_{p^*}\leq 2\inf_{\gamma\in\bbbr}\Vert u-\gamma\Vert_{p^*}$.
This completes the proof in the case $0<p<s$.

\noindent
{\sc The critical case $p=s>1$.}
Note that $k_o=1$. From \eqref{eq9}, H\"older's inequality, and \eqref{eq8}, we get
\begin{align*}
 a_k
 &\leq a_{1}
 +2^{1+1/s}b^{-1/s}\mu(X)^{-1/s}\Vert g\Vert_{s}
 (k-1)^{1/s'}\left(\sum_{i=2}^k2^{is}m_{i-1}\right)^{1/s}
\leq
a_1+C_s b^{-1/s}\Vert g\Vert_sk^{1/s'}.
\end{align*}
Together with \eqref{eq4} this gives
\begin{equation}
\label{eq12}
 a_k\leq C_s b^{-1/s}\Vert g\Vert_{s} k^{1/s'},
 \qquad k\geq k_o.
\end{equation}
Hence for a sufficiently small constant $c_s$,
\[
c_s\left(\frac{b^{1/s}|u-u(x_o)|}{\Vert g\Vert_s}\right)^{s'}
\leq \frac{s\ln 2}{2}\, k
\qquad
\text{on } E_k,\ k\geq 1,
\]
and note that $\exp(\frac{s\ln 2}{2}k)=2^{sk/2}$.

Let $F_{1}=E_{1}$ and $F_k=E_k\setminus E_{k-1}$ for $k>1$. 
Observe that
\[
\mu(F_k)\leq 2^{-s(k-1)}\mu(X),
\qquad
k\geq 1.
\]
Indeed, it is obvious if $k=1$ and for $k>1$ it follows from 
\eqref{eq3}, because $\mu(F_k)\leq m_{k-1}$. Therefore,
\[
\begin{split}
\fint_X\exp\!\left[\frac{c_s}{2^{s'}}\left(\frac{b^{1/s}|u-u_X|}{\Vert g\Vert_s}\right)^{s'}\right]\, d\mu
&\leq
\fint_X\exp\!\left[c_s\left(\frac{b^{1/s}|u-u(x_o)|}{\Vert g\Vert_s}\right)^{s'}\right]\, d\mu\\
&\leq
\mu(X)^{-1}\sum_{k\geq 1} 2^{sk/2}\mu(F_k)\leq C_s,
\end{split}
\]
where the first inequality is a straightforward consequence of the following lemma. 
\begin{lemma}
\label{T6}
If $q>1$ and $C>0$, then
\[
\fint_X\exp\left(\frac{C}{2^q}|u-u_X|^q\right)\, d\mu\leq
\inf_{\gamma\in\bbbr}
\fint_X\exp(C|u-\gamma|^q)\, d\mu.
\]
\end{lemma}
\begin{proof}
Since
\[
|u-u_X|^q\leq 2^{q-1}(|u-\gamma|^q+|\gamma-u_X|^q)\leq
2^{q-1}\left(|u-\gamma|^q+\fint_X|u-\gamma|^q\, d\mu\right),
\]
Jensen's inequality yields
\[
\exp\left(\frac{C}{2^q}|u-u_X|^q\right)\leq 
\exp\left(\frac{C}{2}|u-\gamma|^q\right)
\fint_X\exp\left(\frac{C}{2}|u-\gamma|^q\right)\, d\mu.
\]
Integrating and applying the Cauchy-Schwarz inequality yields
\[
\fint_X
\exp\left(\frac{C}{2^q}|u-u_X|^q\right)\, d\mu\leq
\left[\fint_X\exp\left(\frac{C}{2}|u-\gamma|^q\right)\, d\mu\right]^2
\leq \fint_X \exp(C|u-\gamma|^q)\, d\mu.
\]
The proof of the lemma is complete.
\end{proof}
This also completes the proof in the case $p=s>1$.

\noindent
{\sc The critical case $p=s$, $0<s\leq1$.} 
Now, $k_o$ depends only on $s$. Using
$\sum_j\alpha_j\leq\left(\sum_j\alpha_j^s\right)^{1/s}$
in \eqref{eq9}, and then using \eqref{eq8} and \eqref{eq4}, we obtain
$a_k\leq C_sb^{-1/s}\Vert g\Vert_{s}$ for every $k\geq k_o$.
Since $\bigcup_{k\geq k_o}E_k=X\setminus N$, it follows that $u\in L^\infty(X)$, so $u_X$ is defined and
\[
 \Vert u-u_X\Vert_{\infty}\leq 2\Vert u-u(x_o)\Vert_\infty
 \leq C_sb^{-1/s}\Vert g\Vert_{L^s(X)}.
\]

\noindent
{\sc The supercritical case $p>s$.}
Using \eqref{eq3} in \eqref{eq5}, we obtain
\[
 a_k-a_{k-1}
 \leq C_{s,p}b^{-1/s}\mu(X)^{1/s-1/p}\Vert g\Vert_p
 2^{k(1-p/s)},
 \qquad k>k_o.
\]
Since $p>s$, the geometric series $\sum_{k>k_o}2^{k(1-p/s)}$ converges to a constant that depends on $s$ and $p$ only, because $k_o$ depends only on $p$. Together with \eqref{eq4}, this gives
\[
\Vert u-u_X\Vert_\infty\leq 2\Vert u-u(x_o)\Vert_\infty=
2\sup_{k\geq k_o}a_k
\leq C_{s,p}b^{-1/s}\mu(X)^{1/s-1/p}\Vert g\Vert_p.
\]
\end{proof}

\section{Local estimates}

\begin{proof}[Proof of Theorem~\ref{T2}]
We may assume that $\Vert g\Vert_{L^p(\sigma B)}>0$, since otherwise $u$ is constant almost everywhere on $\sigma B$ and all the conclusions are immediate. Enlarging a null set $N$, if necessary, we may also assume that $g$ is finite on $\sigma B\setminus N$ and that
\[
 |u(x)-u(y)|\leq d(x,y)(g(x)+g(y))
\]
for every $x,y\in\sigma B\setminus N$.

We define the level sets exactly as in the proof of Theorem~\ref{T1}, with $\mu(B)$ in place of $\mu(X)$. Namely, for $k\in\bbbz$ let
\[
 E_k=\left\{x\in\sigma B\setminus N:\
 g(x)\leq 2^k\mu(B)^{-1/p}\Vert g\Vert_{L^p(\sigma B)}\right\},
 \qquad
 m_k=\mu(\sigma B\setminus E_k).
\]
The sets $E_k$ are increasing. Enlarging $N$ by a countable union of null sets, if necessary, we may assume that whenever $m_k=0$, we have $E_k=\sigma B\setminus N$, and hence $E_j=\sigma B\setminus N$ for every $j\geq k$. Since $g$ is finite on $\sigma B\setminus N$,
\[
 \sigma B\setminus N=\bigcup_{k\in\bbbz}E_k.
\]
Chebyshev's inequality gives
\begin{equation}
\label{eq21}
 m_k\leq 2^{-kp}\mu(B).
\end{equation}
Similar to the proof of Theorem~\ref{T1}, we shall need the following estimate:
\begin{equation}
\label{eq22}
 \sum_{k\in\bbbz}2^{kp}m_{k-1}
 \leq \frac{2^p}{1-2^{-p}}\mu(B).
\end{equation}
This is proven using the same Fubini calculation as in \eqref{eq8}, where the integration is over $\sigma B$ and we normalize by $\mu(B)$ instead of $\mu(X)$.

Recall that
\[
 L=\min\{R_o,(\sigma-1)R\}.
\]
Since $L\leq R$ and $L\leq R_o$, applying \eqref{eq20} at the center $z$ of $B$ gives
\begin{equation}
\label{eq23}
 \mu(B)\geq\mu(B(z,L))\geq bL^s.
\end{equation}
Put
\[
 \Lambda_{s,p}=\left(\frac{2^{1/s}}{1-2^{-p/s}}\right)^{s/p},
\]
and let $k_o$ be the smallest integer such that
\[
 \Lambda_{s,p}^{p/s}\,
 b^{-1/s}\mu(B)^{1/s}2^{-k_op/s}\leq L.
\]
Such an integer exists because the left hand side tends to zero as $k_o\to\infty$. By the minimality of $k_o$,
\begin{equation}
\label{eq24}
 \Lambda_{s,p}b^{-1/p}\mu(B)^{1/p}L^{-s/p}
 \leq2^{k_o}
 <2\Lambda_{s,p}b^{-1/p}\mu(B)^{1/p}L^{-s/p}.
\end{equation}
Notice that \eqref{eq23} and the left hand inequality in \eqref{eq24} imply $2^{k_o}>1$, so $k_o\geq1$. Hence, by \eqref{eq21},
\[
 m_{k_o}\leq2^{-k_op}\mu(B)<\mu(B),
\]
and therefore $E_{k_o}\cap B\ne\varnothing$. Fix $x_o\in E_{k_o}\cap B$ and put $\gamma=u(x_o)$. On $E_{k_o}$ the function $u$ is
$2^{k_o+1}\mu(B)^{-1/p}\Vert g\Vert_{L^p(\sigma B)}$-Lipschitz. Since $1<\sigma\leq2$,
$\diam(\sigma B)\leq4R$, and hence
\begin{equation}
\label{eq25}
 a_{k_o}:=\sup_{E_{k_o}}|u-\gamma|
 \leq2^{k_o+3}R\mu(B)^{-1/p}\Vert g\Vert_{L^p(\sigma B)}
 \leq 16\Lambda_{s,p}\frac RL\,b^{-1/p}L^{1-s/p}
 \Vert g\Vert_{L^p(\sigma B)}.
\end{equation}
In the last inequality we used the upper bound in \eqref{eq24}.

For $j>k_o$ define
\[
 r_j=\left(\frac{2m_{j-1}}b\right)^{1/s}.
\]
By \eqref{eq21},
\[
 r_j\leq2^{1/s}b^{-1/s}\mu(B)^{1/s}2^{-(j-1)p/s}.
\]
Therefore, by the definition of $k_o$,
\begin{equation}
\label{eq26}
 \sum_{j=k_o+1}^{\infty}r_j
 \leq\frac{2^{1/s}}{1-2^{-p/s}}\,
 b^{-1/s}\mu(B)^{1/s}2^{-k_op/s}
 \leq L.
\end{equation}
In particular, $r_j\leq L\leq R_o$ for every $j>k_o$.

For $x\in E_k\cap B$, $k>k_o$, we use the construction from
\eqref{eq6}--\eqref{eq7} to choose points $x_k=x$ and $x_j\in E_j$,
$k_o\leq j<k$, with $d(x_j,x_{j-1})\leq r_j$. However, this time we must also verify that every ball used in this construction lies inside $\sigma B$.
Indeed, suppose that $x_j,\ldots,x_k$ have already been chosen. 
If $m_{j-1}=0$, then
$E_{j-1}=E_j=\sigma B\setminus N$ and $r_j=0$, so taking
$x_{j-1}=x_j$ gives the desired step. We may therefore assume
that $m_{j-1}>0$. Then, for $w\in B(x_j,r_j)$, we use $x\in B$ and \eqref{eq26} to estimate,
\[
 d(w,z)<R+\sum_{i=j}^k r_i\leq R+L\leq\sigma R.
\]
Thus $B(x_j,r_j)\subset\sigma B$. Since $0<r_j\leq R_o$, \eqref{eq20} gives
\[
 \mu(B(x_j,r_j))\geq br_j^s=2m_{j-1}>m_{j-1}.
\]
It follows that $B(x_j,r_j)\cap E_{j-1}\neq\varnothing$, and we may therefore choose $x_{j-1}$ in this intersection. This completes the construction.

Since $E_{j-1}\subset E_j$, both $x_j$ and $x_{j-1}$ belong to $E_j$. Hence
\begin{align*}
|u(x_j)-u(x_{j-1})|
&\leq2^{j+1}\mu(B)^{-1/p}\Vert g\Vert_{L^p(\sigma B)}r_j
\\
&=2^{1+1/s}b^{-1/s}\mu(B)^{-1/p}\Vert g\Vert_{L^p(\sigma B)}
2^jm_{j-1}^{1/s}.
\end{align*}
Combining this estimate with the definition of $a_{k_o}$, we obtain the basic estimate
\begin{equation}
\label{eq27}
 \sup_{E_k\cap B}|u-\gamma|
 \leq a_{k_o}
 +2^{1+1/s}b^{-1/s}\mu(B)^{-1/p}\Vert g\Vert_{L^p(\sigma B)}
 \sum_{j=k_o+1}^k2^jm_{j-1}^{1/s},
 \qquad k>k_o.
\end{equation}
This is the local counterpart of the estimate \eqref{eq9}. 
Since \eqref{eq22} provides the same summability
bound as \eqref{eq8}, with $\mu(B)$ in place of $\mu(X)$,
we will refer to the corresponding arguments
in the proof of Theorem~\ref{T1} when estimating the sum
on the right-hand side of \eqref{eq27}.

\medskip
\noindent
{\sc The subcritical case $0<p<s$.}
Let $p^*=sp/(s-p)$. Our first aim is to prove that
\begin{equation}
\label{eq28}
 \sup_{E_k\cap B}|u-\gamma|
 \leq C_{s,p}\frac RL\,b^{-1/s}\mu(B)^{-1/p^*}
 \Vert g\Vert_{L^p(\sigma B)}2^{kp/p^*},
 \qquad k\geq k_o,
\end{equation}
and, when $1\leq p<s$, the sharper estimate
\begin{equation}
\label{eq29}
 \sup_{E_k\cap B}|u-\gamma|
 \leq C_s\left(
 \frac RL+\left(\frac{s-1}{s-p}\right)^{1/s'}
 \right)b^{-1/s}\mu(B)^{-1/p^*}
 \Vert g\Vert_{L^p(\sigma B)}2^{kp/p^*}.
\end{equation}

To estimate the second term on the right-hand side of \eqref{eq27},
we argue as in the proofs of \eqref{eq10} and \eqref{eq15} when
$s>1$, using H\"older's inequality followed by a geometric series
estimate. When $0<s\leq1$, we use the argument in \eqref{eq39}.
In either case, we obtain
\begin{equation}
\label{eq:local-subcritical-increments}
\begin{aligned}
&2^{1+1/s}b^{-1/s}\mu(B)^{-1/p}
 \Vert g\Vert_{L^p(\sigma B)}
 \sum_{j=k_o+1}^k2^jm_{j-1}^{1/s}\\
&\qquad\leq
 C_{s,p}b^{-1/s}\mu(B)^{-1/p^*}
 \Vert g\Vert_{L^p(\sigma B)}2^{kp/p^*}.
\end{aligned}
\end{equation}
As in the proof of \eqref{eq15}, for $1\leq p<s$, the constant $C_{s,p}$ in this bound
satisfies
\begin{equation}
\label{eq:local-subcritical-increments - constant}
 C_{s,p}\leq
 32\left(\frac{s-1}{s-p}\right)^{1/s'}.
\end{equation}

It remains to estimate the term $a_{k_o}$ in \eqref{eq27}.
Note that, in the proof of Theorem~\ref{T1}, $k_o$ depends only on $p$,
whereas here its choice also involves the ratio $\mu(B)/(bL^s)$.
To express the bound for $a_{k_o}$ in \eqref{eq25} in a form
that can be combined with \eqref{eq:local-subcritical-increments},
we use the lower bound in \eqref{eq24}, which gives
\[
 b^{-1/s}\mu(B)^{-1/p^*}2^{k_op/p^*}
 \geq \Lambda_{s,p}^{p/p^*}b^{-1/p}L^{1-s/p}.
\]
Combining this with \eqref{eq25}, and using $k\geq k_o$ and 
$1-p/p^*=p/s$%$\Lambda_{s,p}>1$
, yields
\[
 a_{k_o}\leq 16\Lambda_{s,p}^{p/s}\frac RL\,b^{-1/s}\mu(B)^{-1/p^*}
 \Vert g\Vert_{L^p(\sigma B)}2^{kp/p^*}.
\]
This proves \eqref{eq28}. 
For $1\leq p<s$, we have
\[
 \Lambda_{s,p}^{p/s}
 =\frac{2^{1/s}}{1-2^{-p/s}}
 \leq\frac{2^{1/s}}{1-2^{-1/s}}.
\]
Consequently, the estimate for $a_{k_o}$ holds with some constant $C_s$
that only depends on $s$.
Combining this estimate with
\eqref{eq:local-subcritical-increments} and
\eqref{eq:local-subcritical-increments - constant}
proves \eqref{eq29}.

We now complete the proof of the subcritical embedding. Set
\begin{equation}
\label{eq:F sets}
F_{k_o}=B\cap E_{k_o},
 \qquad
 F_k=B\cap(E_k\setminus E_{k-1}),\quad k>k_o.
\end{equation}
Up to a null set, we have
\[
B=\bigcup_{k\geq k_o}F_k.
\]
 As in the proof of
\eqref{eq11},
\[
 \sum_{k>k_o}2^{kp}\mu(F_k)
 \leq 2^p\mu(B)\Vert g\Vert_{L^p(\sigma B)}^{-p}
       \int_B g^p\,d\mu
 \leq 2^p\mu(B).
\]
For the $k_o$ level, \eqref{eq24} gives instead
\[
 2^{k_op}\mu(F_{k_o})
 \leq 2^p\Lambda_{s,p}^p\frac{\mu(B)^2}{bL^s}.
\]
Together with \eqref{eq23}, these estimates imply
\begin{equation}
\label{eq30}
 \sum_{k\geq k_o}2^{kp}\mu(F_k)
 \leq 2^{p+1}\Lambda_{s,p}^p\frac{\mu(B)^2}{bL^s}.
\end{equation}
As before, when $1\leq p<s$, we can bound the constant $2^{p+1}\Lambda_{s,p}^p$ by some $C_s$.

Summing \eqref{eq28} to the power $p^*$ over the sets $F_k$, exactly as
after \eqref{eq11}, and then using \eqref{eq30}, gives
\[
 \fint_B |u-\gamma|^{p^*}\,d\mu
 \leq C_{s,p}^{p^*}2^{p+1}\Lambda_{s,p}^p\left(\frac RL\right)^{p^*}
 b^{-p^*/s-1}L^{-s}\Vert g\Vert_{L^p(\sigma B)}^{p^*},
\]
where $C_{s,p}$ is the constant in \eqref{eq28}.
Since $1/s+1/p^*=1/p$ and $-s/p^*=1-s/p$, it follows that
\[
 \left(\fint_B|u-\gamma|^{p^*}\,d\mu\right)^{1/p^*}
 \leq C_{s,p}\left(2^{p+1}\Lambda_{s,p}^p\right)^{1/p^*}\frac RL\,b^{-1/p}L^{1-s/p}
 \Vert g\Vert_{L^p(\sigma B)}.
\]
This proves the first assertion in the subcritical case.

For $1\leq p<s$, we repeat this calculation using
\eqref{eq29} in place of \eqref{eq28}.
The factor coming from \eqref{eq30} satisfies
\[
 \left(2^{p+1}\Lambda_{s,p}^p\right)^{1/p^*}
 \leq
 \left(\frac{2^{s+2}}{(1-2^{-1/s})^s}\right)^{1/p^*}
 \leq
 \left(\frac{2^{s+2}}{(1-2^{-1/s})^s}\right)^{1/s'},
\]
where the last inequality uses $p^*\geq s'$.
Thus this factor is bounded independently of $p$ in the
range $1\leq p<s$, and we obtain
\[
 \left(\fint_B|u-\gamma|^{p^*}\,d\mu\right)^{1/p^*}
 \leq C_s\left(
 \frac RL+\left(\frac{s-1}{s-p}\right)^{1/s'}
 \right)b^{-1/p}L^{1-s/p}
 \Vert g\Vert_{L^p(\sigma B)}.
\]
Finally, whenever $p^*\geq1$, the same averaging argument
as in Theorem~\ref{T1} replaces $\gamma$ by $u_B$ at the
cost of a factor of $2$. This completes the proof of the theorem
in the subcritical case.

\medskip
\noindent
{\sc The critical case $p=s>1$.}
Write $s'=s/(s-1)$. Repeating the H\"older estimate leading to
\eqref{eq12}, now starting at the local level $k_o$ instead of
at $1$, and using \eqref{eq22}, \eqref{eq25}, and \eqref{eq27},
we obtain
\begin{equation}
\label{eq31}
 \sup_{E_k\cap B}|u-\gamma|
 \leq C_sb^{-1/s}\Vert g\Vert_{L^s(\sigma B)}
 \left(\frac RL+(k-k_o)^{1/s'}\right),
 \qquad k>k_o.
\end{equation}
The same estimate, with $k=k_o$, follows from \eqref{eq25}.

With the sets $F_k$ defined as in \eqref{eq:F sets}, observe that
since $k_o\geq1$, \eqref{eq21} gives
\begin{equation}
\label{eq32}
 \frac{\mu(F_k)}{\mu(B)}
 \leq2^{-s(k-1)}
 \leq2^{-s(k-k_o)},
 \qquad k>k_o.
\end{equation}
By \eqref{eq31} and $L/R\leq1$, we have
\[
 \left(
 \frac LR\frac{b^{1/s}|u-\gamma|}
 {\Vert g\Vert_{L^s(\sigma B)}}
 \right)^{s'}
 \leq C_s(1+k-k_o)
 \qquad\text{on }F_k,\quad k\geq k_o.
\]
As in the argument following \eqref{eq12}, choosing $c_s>0$
sufficiently small and summing over the sets $F_k$ yields
\begin{equation}
\label{eq33}
 \fint_B
 \exp\!\left[
 c_s\left(
 \frac LR\frac{b^{1/s}|u-\gamma|}
 {\Vert g\Vert_{L^s(\sigma B)}}
 \right)^{s'}
 \right]d\mu
 \leq C_s.
\end{equation}
Indeed, for sufficiently small $c_s$, the integrand in
\eqref{eq33} is at most $C_s2^{s(k-k_o)/2}$ on $F_k$,
$k>k_o$, and is bounded by $C_s$ on $F_{k_o}$ by
\eqref{eq25}. Together with \eqref{eq32}, these bounds
show that the integral average in \eqref{eq33} is at most
\[
 C_s\left(
 1+\sum_{k=k_o+1}^{\infty}
 2^{-s(k-k_o)}2^{s(k-k_o)/2}
 \right)
 =\frac{C_s}{1-2^{-s/2}}.
\]
In particular, $u\in L^1(B)$. The same argument used in the proof of Lemma~\ref{T6} with $X$ replaced by $B$ allows us to replace $\gamma$ by $u_B$ in \eqref{eq33}, after replacing
$c_s$ by $c_s/2^{s'}$. This proves the theorem in the case
$p=s>1$.

\medskip
\noindent
{\sc A pointwise estimate for the remaining cases.}
We first establish a pointwise estimate that will be used to prove uniform
continuity when $p=s\leq1$ and when $p>s$, as well as the H\"older
estimate when $p>s$.

Fix $k\geq k_o$ and $x\in B\setminus N$. 
We now choose a point $x_k\in E_k$ as follows:
If $x\in E_k$, put $x_k=x$.
Otherwise, since $\bigcup_{m\geq k_o}E_m=\sigma B\setminus N$,
we may choose the smallest $m>k$ such that $x\in E_m$. Starting with $x_m=x$,
apply the chain construction down to level $k$, choosing
\[
 x_{j-1}\in B(x_j,r_j)\cap E_{j-1},
 \qquad j=m,m-1,\ldots,k+1,
\]
which gives us a point $x_k\in E_k$. Note that the estimates along the chain give
\[
 d(x,x_k)\leq\sum_{j=k+1}^m r_j
\]
and
\[
 |u(x)-u(x_k)|
 \leq 2\mu(B)^{-1/p}\Vert g\Vert_{L^p(\sigma B)}
       \sum_{j=k+1}^m2^jr_j.
\]
For $x,y\in B\setminus N$, choose $x_k,y_k\in E_k$ in this way.
Since $u|_{E_k}$ is Lipschitz with constant at most
$2^{k+1}\mu(B)^{-1/p}\Vert g\Vert_{L^p(\sigma B)}$, the triangle
inequality gives
\[
\begin{split}
 |u(x)-u(y)|
 &\leq |u(x)-u(x_k)|+|u(y)-u(y_k)|\\
 &\quad+2^{k+1}\mu(B)^{-1/p}\Vert g\Vert_{L^p(\sigma B)}
 \bigl(d(x,x_k)+d(x,y)+d(y,y_k)\bigr).
\end{split}
\]
Using the preceding bounds, together with
$2^k\sum_{j=k+1}^{\infty}r_j\leq\sum_{j=k+1}^{\infty}2^jr_j$, we obtain
\begin{equation}
\label{eq:T2-two-point}
\begin{split}
 |u(x)-u(y)|
 &\leq C_s b^{-1/s}\mu(B)^{-1/p}\Vert g\Vert_{L^p(\sigma B)}
       \sum_{j=k+1}^{\infty}2^jm_{j-1}^{1/s}\\
 &\quad+2^{k+1}\mu(B)^{-1/p}\Vert g\Vert_{L^p(\sigma B)}d(x,y).
\end{split}
\end{equation}
If the series tail in \eqref{eq:T2-two-point} tends to zero as
$k\to\infty$, then $u|_{B\setminus N}$ is uniformly continuous.
Indeed, if this is the case then, for each $\varepsilon>0$, we can choose $k$ so that the first term is less than $\varepsilon/2$, and then take $d(x,y)$ small enough so that the second term is less than $\varepsilon/2$. Since $B\setminus N$ is dense in $B$, this gives a unique uniformly continuous representative on $B$, which we continue to denote by $u$, satisfying \eqref{eq:T2-two-point} for all $x,y\in B$.

\medskip
\noindent
{\sc The critical case $p=s$, $0<s\leq1$.}
As in the corresponding case of Theorem~\ref{T1}, we use
$\ell^s\subset\ell^1$. Together with \eqref{eq22}, this gives
\begin{equation}
\label{eq:T2-critical-tail}
 \sum_{j=k+1}^{\infty}2^jm_{j-1}^{1/s}
 \leq
 \left(\sum_{j=k+1}^{\infty}2^{js}m_{j-1}\right)^{1/s}
 \longrightarrow0
 \qquad\text{as }k\to\infty.
\end{equation}
Thus the preceding argument gives a uniformly continuous
representative on $B$.
Moreover, \eqref{eq22} bounds the right-hand side of
\eqref{eq:T2-critical-tail} at $k=k_o$ by $C_s\mu(B)^{1/s}$.
Combining this bound with \eqref{eq25} and \eqref{eq27}, and using
$R/L\geq1$, yields
\begin{equation}
\label{eq34}
 \sup_{E_k\cap B}|u-\gamma|
 \leq C_s\frac RL\,b^{-1/s}\Vert g\Vert_{L^s(\sigma B)},
 \qquad k\geq k_o.
\end{equation}
Since $B\setminus N=\bigcup_{k\geq k_o}(E_k\cap B)$,
continuity extends this bound to all of $B$. Consequently,
\[
 \operatorname{osc}_B u
 \leq 2\sup_B|u-\gamma|
 \leq C_s\frac RL\,b^{-1/s}\Vert g\Vert_{L^s(\sigma B)}.
\]
This completes the critical case.

\medskip
\noindent
{\sc The supercritical case $p>s$.}
By \eqref{eq21}, for $k\geq k_o$,
\begin{equation}
\label{eq37}
\begin{split}
 \sum_{j=k+1}^{\infty}2^jm_{j-1}^{1/s}
 &\leq 2^{p/s}\mu(B)^{1/s}
        \sum_{j=k+1}^{\infty}2^{j(1-p/s)}\\
 &\leq C_{s,p}\mu(B)^{1/s}2^{k(1-p/s)}.
\end{split}
\end{equation}
Since $1-p/s<0$, these tails tend to zero. Hence the argument
following \eqref{eq:T2-two-point} again gives a continuous
representative of $u$ on $B$, which we continue to denote by $u$.

As in the supercritical case of Theorem~\ref{T1}, we now sum the
increments. Equations \eqref{eq27} and \eqref{eq37} give
\[
 |u(x)-\gamma|\leq a_{k_o}
 +C_{s,p}b^{-1/s}\mu(B)^{1/s-1/p}
 \Vert g\Vert_{L^p(\sigma B)}2^{k_o(1-p/s)},
 \qquad x\in B\setminus N.
\]
The initial term is bounded by \eqref{eq25}. Since $1-p/s<0$,
the lower bound in \eqref{eq24} yields
\[
 b^{-1/s}\mu(B)^{1/s-1/p}2^{k_o(1-p/s)}
 \leq \Lambda_{s,p}^{1-p/s}b^{-1/p}L^{1-s/p}.
\]
Using $R/L\geq1$ and extending by continuity, we obtain
\begin{equation}
\label{eq38}
 |u-\gamma|
 \leq C_{s,p}\frac RL\,b^{-1/p}L^{1-s/p}
 \Vert g\Vert_{L^p(\sigma B)}
 \qquad\text{on }B.
\end{equation}
The stated oscillation estimate now follows from
$\operatorname{osc}_B u\leq2\sup_B|u-\gamma|$.

We now prove the H\"older estimate.
Let $x,y\in B$ be distinct, and set
$c_{s,p}=(2\Lambda_{s,p})^{-p/s}$.
If $d(x,y)\leq c_{s,p}L$, \eqref{eq24} implies that there exists
$k\geq k_o$ such that
\begin{equation}
\label{eq:T2-holder-level}
 2^k\leq b^{-1/p}\mu(B)^{1/p}d(x,y)^{-s/p}<2^{k+1}.
\end{equation}
Using \eqref{eq37} in \eqref{eq:T2-two-point}, we obtain
\begin{equation}
\label{eq:T2-holder-two-terms}
\begin{split}
 |u(x)-u(y)|
 &\leq C_{s,p}b^{-1/s}\mu(B)^{1/s-1/p}
       \Vert g\Vert_{L^p(\sigma B)}2^{k(1-p/s)}\\
 &\quad+2^{k+1}\mu(B)^{-1/p}
       \Vert g\Vert_{L^p(\sigma B)}d(x,y).
\end{split}
\end{equation}
The right-hand inequality in \eqref{eq:T2-holder-level} bounds the
first term in \eqref{eq:T2-holder-two-terms}, since $1-p/s<0$,
while the left-hand inequality bounds the second. Consequently,
\[
 |u(x)-u(y)|
 \leq C_{s,p}b^{-1/p}d(x,y)^{1-s/p}
       \Vert g\Vert_{L^p(\sigma B)}.
\]
If $d(x,y)>c_{s,p}L$, \eqref{eq38} immediately gives
\[
\begin{split}
 |u(x)-u(y)|
 &\leq C_{s,p}\frac RL\,b^{-1/p}L^{1-s/p}
       \Vert g\Vert_{L^p(\sigma B)}\\
 &\leq C_{s,p}\frac RL\,b^{-1/p}d(x,y)^{1-s/p}
       \Vert g\Vert_{L^p(\sigma B)}.
\end{split}
\]
Combining these last two estimates and using $R/L\geq1$ proves the desired
H\"older estimate, and the proof of Theorem~\ref{T2} is complete.
\end{proof}

\section{Poincar\'e inequalities below the critical exponent}

\begin{proof}[Proof of Theorem~\ref{T9}]
Fix a ball $B=B(z,R)$. It is understood that
$\Vert g\Vert_{L^s(10\sigma B)}>0$.

For a locally integrable function $v$ and $T>0$ let
\[
Mv(x)=\sup_{r>0}\fint_{B(x,r)}|v|\,d\mu
\qquad
\text{and}
\qquad
 M_Tv(x)=\sup_{0<r<T}\fint_{B(x,r)}|v|\,d\mu
\]
be the Hardy--Littlewood maximal functions. By \cite[Theorem~3.2]{SMP}, the $(1,p)$-Poincar\'e inequality implies that
\begin{equation}
\label{eq51}
 |u(x)-u(y)|
 \leq C d(x,y)
 \left[
 \big(M_{2\sigma d(x,y)}g^p(x)\big)^{1/p}
 +\big(M_{2\sigma d(x,y)}g^p(y)\big)^{1/p}
 \right]
\end{equation}
for almost every $x,y\in X$ with $x\neq y$.

We apply \eqref{eq51} to points $x,y\in2B$. Since
$d(x,y)\leq 4R$,
every ball occurring in $M_{2\sigma d(x,y)}g^p(x)$ has radius less than $8\sigma R$. Moreover, $x\in2B$, so such a ball is contained in $10\sigma B$. Hence, if
$h=C\big(M(g^p\chi_{10\sigma B})\big)^{1/p}$,
then $h|_{2B}\in D(u|_{2B})$. 

Since $p<s$, we have $s/p>1$. The maximal theorem on doubling spaces, \cite[Theorem~14.13]{SMP}, therefore gives
\begin{equation}
\label{eq52}
\Vert h\Vert_{L^s(2B)}
\leq C
\left\Vert M(g^p\chi_{10\sigma B})\right\Vert_{L^{s/p}(X)}^{1/p}
\leq C
\left\Vert g^p\chi_{10\sigma B}\right\Vert_{L^{s/p}(X)}^{1/p}
 =C\Vert g\Vert_{L^s(10\sigma B)}.
\end{equation}

Let $x\in 2B$ and $0<r\leq 2R$. Then \eqref{eq18} yields
\[
\mu(B(x,r))\geq C\mu(2B)2^{-s}R^{-s}r^s\geq\underbrace{C\mu(B)2^{-s}R^{-s}}_{b}r^s,
\quad
x\in 2B,
\quad
0<r\leq 2R,
\]
which is \eqref{eq20} with $R_o=2R$. Since
\[
L=\min\{2R,(2-1)R\}=R,
\qquad
\frac{L}{R}=1,
\]
Theorem~\ref{T2} yields
\[
\fint_B
\exp\left[
 c_s\left(
 (C\mu(B)2^{-s}R^{-s})^{1/s}
 \frac{|u-u_B|}{\Vert h\Vert_{L^s(2B)}}
 \right)^{\frac{s}{s-1}}
 \right]d\mu
\leq C_s
\]
and \eqref{eq50} follows from \eqref{eq52}.
%It remains to verify the lower bound for the measure that is needed in Theorem~\ref{T2}. Let $x\in2B$ and $0<r\leq2R$. Apply \eqref{eq18} with the ball $2B$, whose radius is $2R$. We obtain
%\[
% \mu(B(x,r))
% \geq C\mu(2B)\left(\frac{r}{2R}\right)^s
% \geq C2^{-s}\frac{\mu(B)}{R^s}r^s.
%\]
%Thus the hypothesis \eqref{eq20} of Theorem~\ref{T2} holds on $2B$ with
%\[
% R_o=2R
% \qquad\text{and}\qquad
% b=C2^{-s}\frac{\mu(B)}{R^s}.
%\]
%We now apply part~(2) of Theorem~\ref{T2} to the ball $B$, with the dilation parameter in that theorem equal to $2$, and with the Haj\l asz gradient $h$. In this case
%\[
% L=\min\{R_o,(2-1)R\}=R.
%\]
%Consequently,
%\[
% \fint_B
% \exp\left[
% c_s\left(
% \frac{b^{1/s}|u-u_B|}
% {\Vert h\Vert_{L^s(2B)}}
% \right)^{s'}
% \right]d\mu
% \leq C_s.
%\]
%Since
%\[
% b^{1/s}=C^{1/s}2^{-1}\frac{\mu(B)^{1/s}}R,
%\]
%the estimate in \eqref{eq52} gives \eqref{eq50}, after changing the constants. This completes the proof.
\end{proof}

\section{Counterexample}

\begin{proof}[Proof of Theorem~\ref{T3}]
Fix $s\in(1,\infty)$. We construct first the metric measure space, then the pair $(u,g)$, and finally verify the Poincar\'e inequality and \eqref{eq40}.

\noindent
{\sc The metric measure space.}
Let
\[
C_0=[0,1],
\quad
C_1=[0,1/3]\cup[2/3,1],
\quad 
\text{and}
\quad
C_j=\frac13 C_{j-1}\cup\left(\frac23+\frac13 C_{j-1}\right),\quad j\geq2.
\]
Then $X=\bigcap_{j=0}^\infty C_j$ is the standard ternary Cantor set.

It is well known that $X$ has Hausdorff dimension $\alpha:=\frac{\log 2}{\log 3}$, and that the space is Ahlfors $\alpha$-regular. If we wrinkle the metric and replace it by $|x-y|^{\alpha/s}$, the space will be Ahlfors $s$-regular. However, we want to work with a different metric $d$ that is bi-Lipschitz equivalent to  $|x-y|^{\alpha/s}$, see \eqref{eq58}. This new metric has additional properties that will greatly simplify computations. 
The metric is an ultrametric, it only has countably many distinct values, and the balls have dyadic structure, see Lemma~\ref{T10}. The dyadic structure, \eqref{eq42}, implies that we only have countably many balls of simple structure and hence checking the $(1,s)$-Poincar\'e inequality is  easy.

Our construction is very simple and although it is known in the vast literature on ultrametric Cantor sets we couldn't find that simple description. 
We would like to popularize it since it really simplifies calculations. This is why we present this construction with details.
We learned the construction from  \cite[Theorem~6.16]{farsi}, but this is not the original source.

For $x\in X$, let $I_j(x)$ denote the unique connected component of $C_j$ containing $x$. Note that the Euclidean length of $I_j(x)$ is $3^{-j}$. If $x,y\in X$, $x\neq y$, we define
\[
 N(x,y)=\min\{j\geq1:I_j(x)\neq I_j(y)\}.
\]
Clearly $3^{-j}\leq |x-y|\leq 3^{-j+1}$ for $j=N(x,y)$. Let
\[
d(x,y)=
\begin{cases}
0,&x=y,\\
\lambda^{N(x,y)-1},&x\neq y,
\end{cases}
\qquad
\text{where}
\qquad
\lambda=2^{-1/s}.
\]
It is easy to see that
\begin{equation}
\label{eq58}
|x-y|^{\alpha/s}\leq d(x,y)\leq 2^{1/s}|x-y|^{\alpha/s}.
\end{equation}

Recall that an ultrametric is a metric that satisfies the inequality
\begin{equation}
\label{eq41}
d(x,z)\leq \max\{d(x,y),d(y,z)\}
\qquad
\text{for all } x,y,z\in X.
\end{equation}
\begin{lemma}
\label{T10}
$(X,d)$ is a compact ultrametric space with $\diam X=1$.
The metric $d$ has only countably many nonzero values
\begin{equation}
\label{eq59}
1,\ \lambda,\ \lambda^2,\ \lambda^3,\ \dots
\end{equation}
For every $x\in X$ and $0<r\leq 1$
\begin{equation}
\label{eq42}
B(x,r)=I_j(x)\cap X
\qquad
\text{and}
\qquad
\diam B(x,r)=\lambda^j,
\end{equation}
where $j\geq 1$ is such that $\lambda^j<r\leq \lambda^{j-1}$.
Moreover every ball is clopen and every nonempty set $V\subset X$ of positive diameter  is contained in a ball of the same diameter.
\end{lemma}
\begin{proof}
Observe that \eqref{eq59} follows from the definition of $d$.
Also, $N(0,1)=1$, so $\diam X=d(0,1)=1$.
Now we will prove that $d$ is an ultrametric.
If $x,y,z\in X$ are pairwise distinct, then
\begin{equation}
\label{eq60}
 N(x,z)\geq \min\{N(x,y),N(y,z)\}.
\end{equation}
Indeed, suppose by way of contradiction that $N(x,z)=j$, but $N(x,y)\geq j+1$ and $N(y,z)\geq j+1$. 
Then $x$ and $y$ belong to the same component of $C_j$, and so do $y$ and $z$. This implies that $x$ and $z$ belong to that same component of $C_j$ and hence $N(x,z)\geq j+1$ which is a contradiction.
This proves \eqref{eq60}. Since $\lambda\in (0,1)$, \eqref{eq41} follows.

In particular $d$ is a metric and \eqref{eq58} shows that $(X,d)$ is homeomorphic to the standard ternary Cantor set so it is compact.

To prove \eqref{eq42}, let $j\geq 1$ be such that $\lambda^j<r\leq\lambda^{j-1}$. We will show that $B(x,r)=I_j(x)\cap X$. If $y\in I_j(x)\cap X$ and $y\neq x$, then $N(x,y)\geq j+1$ and hence $d(x,y)\leq \lambda^j<r$, so $I_j(x)\cap X\subset B(x,r)$.
Conversely, if $y\in B(x,r)$ and $y\neq x$, then $d(x,y)<r\leq \lambda^{j-1}$ and hence $d(x,y)\leq\lambda^j$, because of \eqref{eq59}. This implies that $N(x,y)\geq j+1$, so $y\in I_j(x)$ and hence $B(x,r)\subset I_j(x)\cap X$. This completes the proof that $B(x,r)= I_j(x)\cap X$. The diameter of $I_j(x)\cap X$ equals the distance of its endpoints which equals $\lambda^j$. 

Clearly, if $B(x,r)=X$, it is clopen. If $B(x,r)=I_j(x)\cap X$, $j\geq 1$, then it is at a positive distance from its complement and hence it is clopen too.

Finally, by \eqref{eq59}, $\diam V=\lambda^j$. If $j=0$, $V\subset X$, and $X$ is a ball of diameter $1$. If $j\geq 1$ and $\diam V=\lambda^j<r\leq \lambda^{j-1}$, then
for any $x\in V$ we have
$V\subset B(x,r)=I_j(x)\cap X$ and the ball $B(x,r)$ has diameter $\lambda^j=\diam V$ by \eqref{eq42}.
\end{proof}

Next, we equip $X$ with the measure $\mu$ which is the $s$-Hausdorff measure, associated to the metric $d$. In our definition of the Hausdorff measure we use normalization constant equal to $1$ instead of the usual constant $\omega_s/2^s$.
Note also that since $X$ has no isolated points we can assume that all sets in the coverings in the definition of the Hausdorff measure have positive diameter.

\begin{lemma}
\label{hausd1}
For any $j \geq 0$ and any
$ I_{j}(x) \cap X  $ we have
$$
\mu(I_{j}(x) \cap X) = 2^{-j}.
$$
\end{lemma}
\begin{remark}
Let $x\in X$ and $0<r\leq1=\diam X$. Choose $j\geq1$ such that
$\lambda^j<r\leq\lambda^{j-1}$. Then $B(x,r)=I_j(x)\cap X$ and,
since $\lambda^s=1/2$,
\[
\frac{1}{2} r^s\leq\mu(B(x,r))=2^{-j}=(\lambda^j)^s<r^s.
\]
Thus $\mu$ is Ahlfors $s$-regular.
\end{remark}
\begin{proof}
Clearly, for any $m>j$, $I_j(x)$ contains exactly $2^{m-j}$ connected components $\{ I_{j,k}\}_{k=1}^{2^{m-j}}$  of $C_m$. Since by \eqref{eq42}, the diameter of $I_{j,k}\cap X$ equals $\lambda^m$, the definition of the Hausdorff measure yields that
\[
\mu(I_j(x)\cap X)\leq\lim_{m\to\infty} 2^{m-j}(\lambda^m)^s=2^{-j}.
\]
It remains to show the converse inequality. To prove it, let $\delta\in (0,\lambda^j)$ and take a countable covering
\[
I_j(x)\cap X=\bigcup_{i\in I} V_i,
\qquad
0<\diam V_i<\delta.
\]
Since by Lemma~\ref{T10}, $V_i$ is contained in an open ball of the same diameter, we may assume that the sets $V_i$ are open balls. Then from this open covering of the compact set $I_j(x)\cap X$ we may select a finite sub-covering. Finally, since balls are either disjoint or one is contained in another, we may assume that the balls in the covering are disjoint.
Thus
\begin{equation}
\label{eq61}
I_j(x)\cap X=\bigcup_{k=1}^N I_{jk}\cap X
\end{equation}
and clearly
\[
\sum_{i\in I} (\diam V_i)^s\geq \sum_{k=1}^N \diam (I_{jk}\cap X)^s
\]
Let $\diam(I_{jk}\cap X)=\lambda^{m_k}$, $m_k>j$ and let $m=\max_k m_k$. That is $\lambda^m$ is the smallest diameter of balls in the covering \eqref{eq61}.
Each ball $I_{jk}\cap X$ is the union of $2^{m-m_k}$ disjoint balls of diameter $\lambda^m$ i.e., balls of the form $I_m(x_{jk\ell})\cap X$ for some $x_{jk\ell}\in I_{jk}\cap X$,
\[
I_{jk}\cap X=\bigcup_{\ell=1}^{2^{m-m_k}} I_m(x_{jk\ell})\cap X
\]
and note that
\[
\sum_{\ell=1}^{2^{m-m_k}}\diam (I_m(x_{jk\ell})\cap X)^s
=2^{m-m_k}(\lambda^m)^s=(\lambda^{m_k})^s=\diam(I_{jk}\cap X)^s.
\]
Note also that
\[
I_j(x)\cap X=\bigcup_{k=1}^N \bigcup_{\ell=1}^{2^{m-m_k}} I_m(x_{jk\ell})\cap X
\]
is a covering of $I_j(x)\cap X$ by disjoint balls of $m$th generation, so on the right-hand side we have all $2^{m-j}$ balls of $m$th generation contained in $I_j(x)\cap X$, that is
\[
\sum_{k=1}^N 2^{m-m_k}=2^{m-j}
\qquad
\text{so}
\qquad
\sum_{k=1}^N 2^{-m_k}=2^{-j}
\]
Finally,
\[
\begin{split}
\sum_{i\in I}(\diam V_i)^s
&\geq
\sum_{k=1}^N\diam (I_{jk}\cap X)^s=
\sum_{k=1}^N\sum_{\ell=1}^{2^{m-m_k}}\diam(I_m(x_{jk\ell})\cap X)^s\\
&=
\sum_{k=1}^N 2^{m-m_k}(\lambda^m)^s=\sum_{k=1}^N2^{-m_k}=2^{-j}.
\end{split}
\]
Therefore, taking the infimum over all coverings $\{ V_i\}_{i\in I}$ and letting $\delta\to 0$, yields the desired inequality $\H^s(I_j(x)\cap X)\geq 2^{-j}$. The proof is complete.
\end{proof}

Notice that Lemma~\ref{hausd1} implies that $\mu(X)=1$. Moreover, for any ball $B(x,r)=I_j(x)\cap X$ with $0<r\leq1$ and $j$ satisfying $\lambda^j<r\leq\lambda^{j-1}$, we have
\[
\frac12 r^s\leq\mu(B(x,r))=2^{-j}<r^s,
\]
because $\lambda^s=1/2$. Thus $(X,d,\mu)$ is Ahlfors $s$-regular.

We will use the nested sequence of balls containing $0$. For $j\geq0$, put
\[
B_j=I_j(0)\cap X=[0,3^{-j}]\cap X,
\qquad
A_j=B_j\setminus B_{j+1}.
\]
Thus $B_0=X$, and Lemma~\ref{T10} gives
\[
B_j=B(0,\lambda^{j-1})
\qquad\text{for }j\geq1.
\]
By Lemmata~\ref{T10} and~\ref{hausd1}, for every $j\geq0$,
\begin{equation}
\label{eq44}
\mu(B_j)=2^{-j},\qquad \diam B_j =\lambda^j=2^{-j/s},\qquad \mu(A_j)=2^{-j-1}.
\end{equation}
Since $B_j\downarrow\{0\}$, we also have the disjoint decomposition
\[
X=\{0\}\cup\bigcup_{j\geq0}A_j.
\]
Notice that $\mu(\{0\})=0$, because $\mu(\{0\})\leq\mu(B_j)=2^{-j}$ for every $j\geq0$.

Although $A_j$ is defined as the difference of two balls, it is itself a ball. This observation gives the following description of all balls in $X$.
\begin{lemma}
\label{T8}
For every $j\geq0$ and every $x\in A_j$,
\[
A_j=I_{j+1}(x)\cap X=B(x,\lambda^j)
\qquad\text{and}\qquad
d(x,0)=\lambda^j.
\]
Moreover, every ball containing $0$ is one of the balls $B_j$, and every ball that does not contain $0$ is contained in one of the balls $A_j$.
\end{lemma}
\begin{proof}
The interval $I_j(0)=[0,3^{-j}]$ contains exactly two components of $C_{j+1}$, namely
\[
I_{j+1}(0)=[0,3^{-j-1}]
\qquad\text{and}\qquad
[2\cdot3^{-j-1},3^{-j}].
\]
It follows that
\[
A_j=[2\cdot3^{-j-1},3^{-j}]\cap X=I_{j+1}(x)\cap X
\qquad\text{for every }x\in A_j.
\]
Applying \eqref{eq42} with radius $\lambda^j$ gives $A_j=B(x,\lambda^j)$. Also, $x$ and $0$ belong to the same component of $C_j$ and to different components of $C_{j+1}$, so $N(x,0)=j+1$ and $d(x,0)=\lambda^j$.

Now let $B=B(x,r)$ be any ball in $X$. If $r>1=\diam X$, then $B=X=B_0$. If $0<r\leq1$, choose $m\geq1$ so that $\lambda^m<r\leq\lambda^{m-1}$. By Lemma~\ref{T10}, $B=I_m(x)\cap X$. Therefore, if $0\in B$, then $I_m(x)=I_m(0)$ and $B=B_m$.

Finally, suppose that $0\notin B(x,r)$. The center $x$ is nonzero, so it belongs to exactly one of the sets $A_j$. By the first part of the proof,
\[
d(x,0)=\lambda^j
\qquad\text{and}\qquad
A_j=B(x,\lambda^j).
\]
Since the open ball $B(x,r)$ does not contain $0$, we have $r\leq d(x,0)=\lambda^j$. Consequently,
\[
B(x,r)\subset B(x,\lambda^j)=A_j.
\]
This proves the lemma.
\end{proof}

\noindent
{\sc Construction of the pair $(u,g)$.}
For $k\geq1$ define
\[
 a_k=\frac{\log2}{\log(k+1)},
 \qquad
 U_k=\sum_{i=1}^k a_i,
 \qquad
 U_0=0.
\]
The sequence $(a_k)$ is positive and decreasing, $a_1=1$, and $a_k\to0$. For every $j\geq0$, define $u$ and $g$ on $A_j$ by
\[
 u(x)=U_j,\quad
 g(x)=\bigl(2^{j+1}(a_{j+1}^s-a_{j+2}^s)\bigr)^{1/s}
 \qquad x\in A_j,
\]
and set $u(0)=g(0)=0$. 

We first verify the required integrability. Since $a_k\leq1$, we have $U_j\leq j$. Using \eqref{eq44},
\[
 \Vert u\Vert_{s}^s
 =\sum_{j\geq1}\mu(A_j)U_j^s
 \leq\sum_{j\geq1}2^{-j-1}j^s<\infty.
\]
Also, since $a_k\to0$, the series below telescopes and gives
\begin{equation}
\label{eq45}
 \Vert g\Vert_{s}^s
 =\sum_{j\geq0}\mu(A_j)2^{j+1}(a_{j+1}^s-a_{j+2}^s)
 =\sum_{j\geq0}(a_{j+1}^s-a_{j+2}^s)=a_1^s=1.
\end{equation}
Thus, $u,g\in L^s(X)$ and $\Vert g\Vert_{s}=1$.

\noindent
{\sc The Poincar\'e inequality.}
We claim that for every ball $B\subset X$,
\begin{equation}
\label{eq46}
 \fint_B|u-u_B|\,d\mu
 \leq2\diam B\left(\fint_Bg^s\,d\mu\right)^{1/s}.
\end{equation}
If $0\notin B$, then Lemma~\ref{T8} gives $B\subset A_j$ for some $j\geq0$. The function $u$ is constant on $A_j$, so the left hand side of \eqref{eq46} is zero.

It remains to consider balls containing $0$, that is, $B=B_j$ for some $j\geq0$. If $x\in A_n$ with $n\geq j$, then $u(x)-U_j=U_n-U_j=\sum_{k=j+1}^n a_k$. Also, $x\in B_k$ precisely when $k\leq n$. Thus, for $x\in B_j\setminus\{0\}$,
\[
 u(x)-U_j=\sum_{k=j+1}^\infty a_k\chi_{B_k}(x).
\]
In particular, $u\geq U_j$ almost everywhere on $B_j$. Since the summands are nonnegative and $\mu(\{0\})=0$, Fubini's theorem and \eqref{eq44} yield
\begin{equation}
\label{eq47}
 \fint_{B_j}|u-U_j|\,d\mu
 =\sum_{k=j+1}^\infty a_k\frac{\mu(B_k)}{\mu(B_j)}
 =\sum_{k=j+1}^\infty a_k2^{j-k}
 \leq a_{j+1}.
\end{equation}
Here we used that $(a_k)$ is decreasing and $\sum_{k=j+1}^\infty2^{j-k}=1$.

On the other hand, $B_j\setminus\{0\}$ is the disjoint union of the sets $A_n$, $n\geq j$. Using \eqref{eq44} and the definition of $g$, we obtain
\[
 \int_{B_j}g^s\,d\mu
 =\sum_{n\geq j}\int_{A_n}g^s\,d\mu
 =\sum_{n\geq j}(a_{n+1}^s-a_{n+2}^s)=a_{j+1}^s.
\]
Therefore, again by \eqref{eq44},
\begin{equation}
\label{eq48}
 \diam B_j\left(\fint_{B_j}g^s\,d\mu\right)^{1/s}
 =2^{-j/s}\bigl(2^ja_{j+1}^s\bigr)^{1/s}=a_{j+1}.
\end{equation}
Finally, the triangle inequality gives
\[
\fint_{B_j}|u-u_{B_j}|\,d\mu
\leq\fint_{B_j}|u-U_j|\,d\mu+|u_{B_j}-U_j|
\leq2\fint_{B_j}|u-U_j|\,d\mu.
\]
Together with \eqref{eq47} and \eqref{eq48}, this proves \eqref{eq46}. Hence, $(u,g)$ satisfies a $(1,s)$-Poincar\'e inequality with parameters $\sigma=1$ and $C_{PI}=2$.

\noindent
{\sc Failure of exponential integrability.}
It remains to prove \eqref{eq40}. Fix $q>1$ and $c>0$. 
Let $j_o\geq 1$ be such that  $|U_j-u_X|\geq U_j2^{-1/q}$ for $j\geq j_o$. Existence of $j_o$ follows from the fact that $U_j\geq j\log2/\log(j+1)\to\infty$. Then
\[
\begin{split}
\int_X\exp\bigl(c|u-u_X|^q\bigr)\,d\mu
&=
\sum_{j\geq0}\mu(A_j)\exp\bigl(c|U_j-u_X|^q\bigr)\geq   
\sum_{j\geq j_o} 2^{-j-1}\exp\Big(\frac{c}{2}U_j^q\Big)\\
&\geq
\frac{1}{2}\sum_{j\geq j_o} \exp\Big(-j\log 2+\frac{c}{2}\Big(\frac{j\log2}{\log(j+1)}\Big)^q\Big)=\infty.
\end{split}
\]
The proof is complete.
\end{proof}

\end{document}